\documentclass{article}
\usepackage[utf8]{inputenc}

\usepackage{amsmath, amsthm, amssymb, mathrsfs, thmtools, enumerate}
\usepackage[colorlinks = true, linkcolor = black, urlcolor = blue, citecolor = black, pdfborder={0 0 0}]{hyperref} 
\usepackage{tikz} 

\usepackage{xcolor}
\definecolor{junglegreen}{RGB}{65,150, 00}

\allowdisplaybreaks[4]
\declaretheoremstyle[
  spaceabove=1em, spacebelow=1em,
  headfont= \bfseries,
  notefont=\mdseries, notebraces={(}{)},
  bodyfont=\normalfont,
  postheadspace=1em,
  qed= \tiny $\blacksquare$
]{example}

\declaretheoremstyle[
  spaceabove=\topsep, spacebelow=\topsep,
  headfont= \bfseries,
  notefont=\mdseries, notebraces={(}{)},
  bodyfont=\normalfont,
  postheadspace=1em,
  qed= $\blacktriangledown$
]{remark}

\declaretheoremstyle[
  spaceabove=\topsep, spacebelow=\topsep,
  headfont= \bfseries,
  notefont=\mdseries, notebraces={(}{)},
  bodyfont=\it{\normalfont},
  postheadspace=1em,
  qed= \qedsymbol
]{obvious}

\theoremstyle{plain}
\newtheorem{Thm}{Theorem}[section]

\newtheorem{Lem}[Thm]{Lemma}

\theoremstyle{definition}
\newtheorem{Def}[Thm]{Definition}

\declaretheorem[style=example, name = Example, sharenumber = Thm]{Exam}

\numberwithin{equation}{section}
\numberwithin{table}{section}

\newcommand{\N}{\mathbb{N}}

\newcommand{\Z}{\mathbb{Z}}
\newcommand{\Q}{\mathbb{Q}}
\newcommand{\QG}{\Q[G]}
\newcommand{\QH}{\Q[H]}

\renewcommand{\S}{\mathfrak{S}}
\newcommand{\T}{\mathfrak{T}}
\newcommand{\z}{\mathcal{Z}}
\renewcommand{\H}{\mathcal{H}}
\renewcommand{\L}{\mathcal{L}}
\newcommand{\K}{\mathcal{K}}
\newcommand{\D}{\mathcal{D}}
\newcommand{\dsum}{\displaystyle\sum}

\DeclareMathOperator{\Aut}{Aut}
\DeclareMathOperator{\Span}{Span}

\newcommand{\thmref}[1]{Theorem \ref{#1}}

\newcommand{\lemref}[1]{Lemma \ref{#1}}

\newcommand{\defref}[1]{Definition \ref{#1}}
\newcommand{\examref}[1]{Example \ref{#1}}

\newcommand{\secref}[1]{Section \ref{#1}}

\begin{document}

\title{Enumeration Techniques on Cyclic Schur Rings}
\author{Andrew Misseldine\footnote{Andrew Misseldine, Southern Utah University, andrewmisseldine@suu.edu, (telephone) 435-865-8228, (fax) 435-865-8666}}
\date{17 September 2021}

\maketitle

\begin{abstract}
{Any Schur ring is uniquely determined by a partition of the elements of the group. In this paper we present a general technique for enumerating Schur rings over cyclic groups using traditional Schur rings. We also survey recent efforts to enumerate Schur rings over cyclic groups of specific orders. 
}



\end{abstract}

\section{Introduction}
The Bell numbers $B(n)$ count the number of ways a finite set $X$  can be partitioned when $|X|=n$. Without the loss of generality, we may assume $X=\{0, 1, 2, \ldots, n-1\}$. Hence, the Bell numbers count the number of \emph{set partitions} of $X$. The first few Bell numbers are given as $1, 1, 2, 5, 15, 52, 203, 877, 4140, \ldots$\footnote{Sequence A000110 in OEIS.} Could a similar counting problem be considered for finite groups? That is, could one construct a sequence of numbers which enumerate the group partitions of a finite group $G$ when $|G|=n$? We would only want to count partitions that in some way ``respect'' the group structure. Some examples of group-theoretic partitions could be conjugacy classes, automorphism classes, inverse pairs, (double) cosets, or membership of a subgroup, to name a few possibilities. All of these examples of group-theoretic partitions are, in fact, examples of Schur rings (see \defref{def:schurring}). 

Schur rings were first introduced by Wielandt \cite{Wielandt49} as a tool to study permutation groups and are based upon a method originated by Schur \cite{Schur33}. See \cite{Wielandt64} for a detailed treatment of this ``Method of Schur.'' As Schur rings themselves are subrings of groups algebras first considered by Schur, their name is fitting. On the other hand, Schur rings are constructed using partitions of the group which themselves satisfy the group axioms, that is, Schur partitions are partitions of a group which are closed under identity, inverses, and multiplication. Hence, Schur partitions, that is, those partitions of a group which afford a Schur ring, are the natural candidate to be the group-theoretic partition we seek to count in this paper. 

In the category of sets, the only invariant of a set is its cardinality. Hence, the Bell numbers are indexed by natural numbers. Our counting of Schur rings over a group should likewise be indexed by its isomorphism type. The family of cyclic groups is a natural place to begin this enumeration since the set $X$ can always be given the structure of a cyclic group. Let $\z_n=\langle z\rangle$ denote the cyclic group of order $n$, written multiplicatively.  Schur rings over cyclic groups have been of great interest for the last few decades because of their connection to algebraic graph theory (see \cite{Muzychuk09}). Cyclic groups are also among a small set of group families that currently possess a classification theorem of Schur rings (see \thmref{thm:fundamental}). For these reasons, we consider the sequence of numbers $\Omega(n)$ which denote the number of Schur rings over the cyclic group $\z_n$.

We will now establish some notation for this paper. Let $G$ be a finite group, and let $\QG$ denote its rational group algebra. Let $\L(G)$ denote the lattice of subgroups of $G$. For any subset $C\subseteq G$, we may identify this subset with an element of the group algebra $\QG$, namely, $\sum_{g\in C} g\in \QG$. Such an element is called a \emph{simple quantity}, and when there is no confusion we will denote simple quantities by the subsets themselves. Define 
$C^* := \{x^{-1} \mid x\in C\}$ for all 
$C\subseteq G$. 

\begin{Def}\label{def:schurring} Let $\{C_1, C_2, \ldots, C_r\}$ be a partition of $G$, and let $\S$ be the subspace of $\QG$ spanned by the simple quantities $C_1, C_2,\ldots C_r$. We say that $\S$ is a \emph{Schur ring}\label{def:schurring} over $G$ if 
\begin{enumerate}
\item $C_1 = \{1\}$, 
\item For each $i$, there is a $j$ such that $C_i^* = C_j$,
\item For each $i$ and $j$, $C_iC_j = \dsum_{k=1}^r \lambda_{ijk}C_k$, for $\lambda_{ijk}\in \N$.
\end{enumerate}
The sets $C_1, C_2, \ldots, C_r$ are called the \emph{primitive sets} of $\S$ or the \emph{$\S$-classes}
.
\end{Def}

Note that a Schur ring $\S$ is uniquely determined by its associated partition of $G$. We will denote this partition as $\D(\S)$. 

The first attempt to classify Schur rings began with Schur himself. In this original setting, the \emph{transitivity module} of a permutation group $G$ acting on a regular subgroup $H$ is the subspace of $\Q[H]$ spanned by the $G_e$-orbits, where $G_e$ is the stabilizer of the identity $e\in H$. Transitivity modules are always Schur rings. Schur conjectured that all Schur rings over $H$ were transitivity modules for some $G\le \text{Sym}(H)$. Such a Schur ring is called \emph{Schurian}, first coined by P\"oschel \cite{Poschel74}, and a group $H$ is likewise called \emph{Schurian}  if all Schur rings over $H$ are Schurian. Wielandt \cite{Wielandt64} disproved this conjecture. As Schur himself mostly worked with cyclic Schur rings, another conjecture suggested that all cyclic groups were Schurian. This conjecture was likewise disproved by Evdokimov and Ponomarenko in \cite{Ponomarenko01}. Evdokimov, Kov\'acs, and Ponomarenko later classified all Schurian cyclic groups in \cite{Ponomarenko13} as those cyclic groups whose orders\footnote{We point out that all the cyclic group examples considered in \secref{sec:proof} are Schurian by this theorem.} are $p^k, pq^k, 2pq^k, pqr, 2pqr$, where $p, q, r$ are distinct primes. See \cite{Ponomarenko16} for a history of efforts made to classify Schurian cyclic and non-cyclic groups. 

The first attempt to enumerate Schur rings began with Liskovets and P\"oschel \cite{Liskovets} who enumerated wreath-indecomposable Schur rings over the cyclic group of order $p^n$, where $p$ is an odd prime, in an attempt to count certain circulant graphs. Kov\'acs \cite{Kovacs} later solves the more difficult problem when $p=2$. In \cite{Counting}, the author provides recursive formulas to count the number of Schur rings over a cyclic group of order $p^n$, where $p$ is an arbitrary prime. In \cite{CountingIII}, Keller, the author, and Sullivan enumerate all Schur rings over cyclic groups of order $pq$ and $4p$, where $p$ and $q$ are distinct odd primes. In \cite{Wagner}, Humphries and Wagner enumerate symmetric Schur ring over cyclic groups and show these Schur rings are in one-to-one or one-to-two correspondence with central Schur rings of projective special linear groups. In \cite{Ziv14}, Ziv-Av enumerates all Schur rings over small finite groups up to order 63 using computer software. 

In this paper, we generalize the enumeration techniques illustrated introduced in \cite{Counting}, as well as survey the enumerations found in \cite{Counting} and \cite{CountingIII}. Also, using the computer software MAGMA \cite{Magma}, we enumerate the number of all Schur rings over cyclic groups of orders up to 400 (see Table \ref{tab:NumberSRings}).

\begin{table}[p]
 \centering
  \caption{The Number of Schur Rings over $\z_n$}
\resizebox{0.83\textwidth}{!}{
    \begin{tabular}{|c|c||c|c||c|c||c|c||c|c||c|c||c|c||c|c||c|c||c|c|}
    \hline
    $n$ & $\Omega(n)$ & $n$ & $\Omega(n)$ & $n$ & $\Omega(n)$ & $n$ & $\Omega(n)$ & $n$ & $\Omega(n)$ & $n$ & $\Omega(n)$ & $n$ & $\Omega(n)$ & $n$ & $\Omega(n)$ & $n$ & $\Omega(n)$ & $n$ & $\Omega(n)$ \\ \hline\hline
    \textbf{1} & 1     & \textbf{41} & 8     & \textbf{81} & 92    & \textbf{121} & 21    & \textbf{161} & 53    & \textbf{201} & 53    & \textbf{241} & 20    & \textbf{281} & 16    & \textbf{321} & 27    & \textbf{361} & 43 \\
    \hline
    \textbf{2} & 1     & \textbf{42} & 188   & \textbf{82} & 25    & \textbf{122} & 37    & \textbf{162} & 1224  & \textbf{202} & 28    & \textbf{242} & 128   & \textbf{282} & 188   & \textbf{322} & 360   & \textbf{362} & 55 \\
    \hline
    \textbf{3} & 2     & \textbf{43} & 8     & \textbf{83} & 4     & \textbf{123} & 55    & \textbf{163} & 10    & \textbf{203} & 81    & \textbf{243} & 345   & \textbf{283} & 8     & \textbf{323} & 103   & \textbf{363} & 289 \\
    \hline
    \textbf{4} & 3     & \textbf{44} & 61    & \textbf{84} & 1397  & \textbf{124} & 119   & \textbf{164} & 121   & \textbf{204} & 1863  & \textbf{244} & 179   & \textbf{284} & 119   & \textbf{324} & 15934 & \textbf{364} & 5147 \\
    \hline
    \textbf{5} & 3     & \textbf{45} & 140   & \textbf{85} & 60    & \textbf{125} & 58    & \textbf{165} & 670   & \textbf{205} & 93    & \textbf{245} & 450   & \textbf{285} & 997   & \textbf{325} & 659   & \textbf{365} & 139 \\
    \hline
    \textbf{6} & 7     & \textbf{46} & 13    & \textbf{86} & 25    & \textbf{126} & 2099  & \textbf{166} & 13    & \textbf{206} & 25    & \textbf{246} & 380   & \textbf{286} & 546   & \textbf{326} & 31    & \textbf{366} & 558 \\
    \hline
    \textbf{7} & 4     & \textbf{47} & 4     & \textbf{87} & 41    & \textbf{127} & 12    & \textbf{167} & 4     & \textbf{207} & 177   & \textbf{247} & 153   & \textbf{287} & 109   & \textbf{327} & 81    & \textbf{367} & 8 \\
    \hline
    \textbf{8} & 10    & \textbf{48} & 1033  & \textbf{88} & 334   & \textbf{128} & 2989  & \textbf{168} & 12494 & \textbf{208} & 3256  & \textbf{248} & 658   & \textbf{288} & 218905 & \textbf{328} & 694   & \textbf{368} & 2030 \\
    \hline
    \textbf{9} & 7     & \textbf{49} & 21    & \textbf{89} & 8     & \textbf{129} & 53    & \textbf{169} & 43    & \textbf{209} & 79    & \textbf{249} & 27    & \textbf{289} & 31    & \textbf{329} & 53    & \textbf{369} & 373 \\
    \hline
    \textbf{10} & 10    & \textbf{50} & 79    & \textbf{90} & 1581  & \textbf{130} & 457   & \textbf{170} & 411   & \textbf{210} & 8339  & \textbf{250} & 558   & \textbf{290} & 457   & \textbf{330} & 8339  & \textbf{370} & 679 \\
    \hline
    \textbf{11} & 4     & \textbf{51} & 35    & \textbf{91} & 97    & \textbf{131} & 8     & \textbf{171} & 283   & \textbf{211} & 16    & \textbf{251} & 8     & \textbf{291} & 83    & \textbf{331} & 16    & \textbf{371} & 81 \\
    \hline
    \textbf{12} & 32    & \textbf{52} & 91    & \textbf{92} & 61    & \textbf{132} & 1397  & \textbf{172} & 119   & \textbf{212} & 91    & \textbf{252} & 23526 & \textbf{292} & 180   & \textbf{332} & 61    & \textbf{372} & 2745 \\
    \hline
    \textbf{13} & 6     & \textbf{53} & 6     & \textbf{93} & 53    & \textbf{133} & 99    & \textbf{173} & 6     & \textbf{213} & 53    & \textbf{253} & 53    & \textbf{293} & 6     & \textbf{333} & 442   & \textbf{373} & 12 \\
    \hline
    \textbf{14} & 13    & \textbf{54} & 232   & \textbf{94} & 13    & \textbf{134} & 25    & \textbf{174} & 284   & \textbf{214} & 13    & \textbf{254} & 37    & \textbf{294} & 3327  & \textbf{334} & 13    & \textbf{374} & 467 \\
    \hline
    \textbf{15} & 21    & \textbf{55} & 41    & \textbf{95} & 61    & \textbf{135} & 854   & \textbf{175} & 353   & \textbf{215} & 81    & \textbf{255} & 1051  & \textbf{295} & 41    & \textbf{335} & 81    & \textbf{375} & 1464 \\
    \hline
    \textbf{16} & 37    & \textbf{56} & 334   & \textbf{96} & 6719  & \textbf{136} & 442   & \textbf{176} & 3030  & \textbf{216} & 26202 & \textbf{256} & 14044 & \textbf{296} & 766   & \textbf{336} & 126762 & \textbf{376} & 334 \\
    \hline
    \textbf{17} & 5     & \textbf{57} & 40    & \textbf{97} & 12    & \textbf{137} & 8     & \textbf{177} & 27    & \textbf{217} & 125   & \textbf{257} & 9     & \textbf{297} & 1063  & \textbf{337} & 20    & \textbf{377} & 133 \\
    \hline
    \textbf{18} & 42    & \textbf{58} & 19    & \textbf{98} & 128   & \textbf{138} & 188   & \textbf{178} & 25    & \textbf{218} & 37    & \textbf{258} & 366   & \textbf{298} & 19    & \textbf{338} & 262   & \textbf{378} & 20441 \\
    \hline
    \textbf{19} & 6     & \textbf{59} & 4     & \textbf{99} & 177   & \textbf{139} & 8     & \textbf{179} & 4     & \textbf{219} & 82    & \textbf{259} & 153   & \textbf{299} & 81    & \textbf{339} & 69    & \textbf{379} & 16 \\
    \hline
    \textbf{20} & 47    & \textbf{60} & 1103  & \textbf{100} & 563   & \textbf{140} & 2142  & \textbf{180} & 17888 & \textbf{220} & 2142  & \textbf{260} & 3590  & \textbf{300} & 24672 & \textbf{340} & 3275  & \textbf{380} & 3181 \\
    \hline
    \textbf{21} & 27    & \textbf{61} & 12    & \textbf{101} & 9     & \textbf{141} & 27    & \textbf{181} & 18    & \textbf{221} & 119   & \textbf{261} & 275   & \textbf{301} & 125   & \textbf{341} & 145   & \textbf{381} & 79 \\
    \hline
    \textbf{22} & 13    & \textbf{62} & 25    & \textbf{102} & 243   & \textbf{142} & 25    & \textbf{182} & 658   & \textbf{222} & 421   & \textbf{262} & 25    & \textbf{302} & 37    & \textbf{342} & 3168  & \textbf{382} & 25 \\
    \hline
    \textbf{23} & 4     & \textbf{63} & 187   & \textbf{103} & 8     & \textbf{143} & 81    & \textbf{183} & 81    & \textbf{223} & 8     & \textbf{263} & 4     & \textbf{303} & 61    & \textbf{343} & 113   & \textbf{383} & 4 \\
    \hline
    \textbf{24} & 172   & \textbf{64} & 657   & \textbf{104} & 514   & \textbf{144} & 21451 & \textbf{184} & 334   & \textbf{224} & 13299 & \textbf{264} & 12494 & \textbf{304} & 3027  & \textbf{344} & 658   & \textbf{384} & 319416 \\
    \hline
    \textbf{25} & 13    & \textbf{65} & 67    & \textbf{105} & 670   & \textbf{145} & 67    & \textbf{185} & 100   & \textbf{225} & 2096  & \textbf{265} & 67    & \textbf{305} & 133   & \textbf{345} & 670   & \textbf{385} & 1318 \\
    \hline
    \textbf{26} & 19    & \textbf{66} & 147   & \textbf{106} & 19    & \textbf{146} & 37    & \textbf{186} & 366   & \textbf{226} & 31    & \textbf{266} & 672   & \textbf{306} & 2683  & \textbf{346} & 19    & \textbf{386} & 43 \\
    \hline
    \textbf{27} & 25    & \textbf{67} & 8     & \textbf{107} & 4     & \textbf{147} & 289   & \textbf{187} & 69    & \textbf{227} & 4     & \textbf{267} & 55    & \textbf{307} & 12    & \textbf{347} & 4     & \textbf{387} & 369 \\
    \hline
    \textbf{28} & 61    & \textbf{68} & 77    & \textbf{108} & 2219  & \textbf{148} & 135   & \textbf{188} & 61    & \textbf{228} & 2071  & \textbf{268} & 119   & \textbf{308} & 2715  & \textbf{348} & 2157  & \textbf{388} & 181 \\
    \hline
    \textbf{29} & 6     & \textbf{69} & 27    & \textbf{109} & 12    & \textbf{149} & 6     & \textbf{189} & 1225  & \textbf{229} & 12    & \textbf{269} & 6     & \textbf{309} & 53    & \textbf{349} & 12    & \textbf{389} & 6 \\
    \hline
    \textbf{30} & 147   & \textbf{70} & 281   & \textbf{110} & 281   & \textbf{150} & 2124  & \textbf{190} & 415   & \textbf{230} & 281   & \textbf{270} & 14283 & \textbf{310} & 549   & \textbf{350} & 4128  & \textbf{390} & 14253 \\
    \hline
    \textbf{31} & 8     & \textbf{71} & 8     & \textbf{111} & 61    & \textbf{151} & 12    & \textbf{191} & 8     & \textbf{231} & 839   & \textbf{271} & 16    & \textbf{311} & 8     & \textbf{351} & 1949  & \textbf{391} & 69 \\
    \hline
    \textbf{32} & 151   & \textbf{72} & 2311  & \textbf{112} & 2030  & \textbf{152} & 496   & \textbf{192} & 45694 & \textbf{232} & 514   & \textbf{272} & 2872  & \textbf{312} & 20014 & \textbf{352} & 13299 & \textbf{392} & 7641 \\
    \hline
    \textbf{33} & 27    & \textbf{73} & 12    & \textbf{113} & 10    & \textbf{153} & 238   & \textbf{193} & 14    & \textbf{233} & 8     & \textbf{273} & 1611  & \textbf{313} & 16    & \textbf{353} & 12    & \textbf{393} & 53 \\
    \hline
    \textbf{34} & 16    & \textbf{74} & 28    & \textbf{114} & 277   & \textbf{154} & 360   & \textbf{194} & 37    & \textbf{234} & 3267  & \textbf{274} & 25    & \textbf{314} & 37    & \textbf{354} & 188   & \textbf{394} & 28 \\
    \hline
    \textbf{35} & 41    & \textbf{75} & 185   & \textbf{115} & 41    & \textbf{155} & 81    & \textbf{195} & 1142  & \textbf{235} & 41    & \textbf{275} & 395   & \textbf{315} & 8717  & \textbf{355} & 81    & \textbf{395} & 81 \\
    \hline
    \textbf{36} & 284   & \textbf{76} & 90    & \textbf{116} & 91    & \textbf{156} & 2157  & \textbf{196} & 904   & \textbf{236} & 61    & \textbf{276} & 1397  & \textbf{316} & 119   & \textbf{356} & 121   & \textbf{396} & 22162 \\
    \hline
    \textbf{37} & 9     & \textbf{77} & 53    & \textbf{117} & 291   & \textbf{157} & 12    & \textbf{197} & 9     & \textbf{237} & 53    & \textbf{277} & 12    & \textbf{317} & 6     & \textbf{357} & 1152  & \textbf{397} & 18 \\
    \hline
    \textbf{38} & 19    & \textbf{78} & 284   & \textbf{118} & 13    & \textbf{158} & 25    & \textbf{198} & 1989  & \textbf{238} & 467   & \textbf{278} & 25    & \textbf{318} & 284   & \textbf{358} & 13    & \textbf{398} & 37 \\
    \hline
    \textbf{39} & 41    & \textbf{79} & 8     & \textbf{119} & 69    & \textbf{159} & 41    & \textbf{199} & 12    & \textbf{239} & 8     & \textbf{279} & 369   & \textbf{319} & 81    & \textbf{359} & 4     & \textbf{399} & 1601 \\
    \hline
    \textbf{40} & 262   & \textbf{80} & 1646  & \textbf{120} & 10130 & \textbf{160} & 11256 & \textbf{200} & 4973  & \textbf{240} & 107165 & \textbf{280} & 19935 & \textbf{320} & 80768 & \textbf{360} & 257731 & \textbf{400} & 51694 \\\hline
    \end{tabular}
  \label{tab:NumberSRings}
    }
\end{table}

In terms of enumerating Schur rings over $G$, Schurian rings are somewhat problematic because they depend upon subgroups of $\text{Sym}(G)$ which are external to $G$ and do not lead to recursive methods. Instead, this paper will employ an alternative classification of Schur rings given by Leung and Man \cite{LeungII, LeungI}, which were coined \emph{traditional Schur rings} in \cite{InfiniteI} (see \secref{sec:Schur}). In the 1990's, Leung, Man, and others (e.g., \cite{Leung90, Muzychuk93, Muzychuk94, Muzychuk98}) investigated how cyclic Schur rings could be classified internally, namely using subgroups and quotients. This internal classification\footnote{It should be noted that the classification of Schurian groups does not necessarily inhibit enumeration. Lang in \cite{Lang18} using the classificiation of Schurian abelian groups in \cite{Ponomarenko16} to enumerate supercharacter theories over groups of the form $\z_p\times \z_2\times \z_2$. Hendrickson \cite{Hendrickson} showed that supercharacter theories of a group are equivalent to the central Schur rings of that group. Much work of late has been made to classify supercharacter theories,  e.g., \cite{Ash18, Ash19, Lewis17} which is parallel to these efforts to classify and enumerate Schur rings.} of Schur rings better allows for recursion in our enumeration and will be the foundation of our technique. 

The general technique of counting Schur rings over cyclic groups comes from the following basic strategy. By the Fundamental Theorem of Schur Rings over Cyclic Groups proven by Leung and Man (see \thmref{thm:fundamental}), all Schur rings over cyclic groups are traditional. The MAGMA code recursively enumerates all traditional Schur rings over the subquotients of $\z_n$ and checks for \emph{wedge-compatibility}. \emph{Indecomposable} Schur rings are classified and wedge-compatiability is considered to develop recursive relations on the number of Schur rings over subquotients of $\z_n$.  From here, wedge and direct products of Schur rings are considered using this recursion. In regard to proving counting formulas, a similar approach is taken. As this manuscript is primarily survey in nature, many details are omitted for the sake of brevity, but great efforts have also been made to include sufficient details to make this enumeration technique self-contained. The reader should consult the extensive bibliography for further details.

The author wants to personally thank Stephen P. Humphries for his very appreciated technical assistance rendered for the enumeration of $\z_{288}$ and all orders above $300$. The author would like also to thank Brent Kerby whose code in \cite{Kerby} offered inspiration of the code used herein. Finally, the author wants to thank the anonymous referee whose helpful comments significantly improved this manuscript.

\section{Traditional Schur Rings}\label{sec:Schur}
We begin with examples of Schur rings that are important for classifying Schur rings over cyclic groups.  

Given a finite group $G$, the partition of the group into singletons, that is, $\{\{g\}\mid g\in G\}$, affords a Schur ring structure called the \emph{discrete Schur ring}. Note that this is just the group algebra $\QG$ itself. As the coefficient ring will play little role here, we will abuse notation by using $G$ to both denote the group $G$ and the group ring $\QG$, when the context is clear. Another example is the partition $\{1, G\smallsetminus 1\}$. We call this Schur ring the \emph{trivial Schur ring} and denote it as $G^0$. For any finite group, the discrete and trivial Schur rings are always available. In the case of $\z_1$ and $\z_2$, they are one and the same (in fact, there is only one Schur ring over both $\z_1$ and $\z_2$). Otherwise, they are distinct.

Let $\Aut(G)$ be the automorphism group of $G$, and let $\H\le \Aut(G)$. Then $\H$ partitions $G$ according to the automorphic action of $\H$ on $G$, via the $\H$-orbits. This Schur ring is called an \emph{automorphic Schur ring}\footnote{Automorphic Schur ring are also commonly referred to as \emph{orbit Schur rings}, as was the case in \cite{Counting}} and is denoted $G^\H$. Note that the discrete Schur ring is automorphic where $\H=1$. The center of the group algebra, namely $Z(\QG)$, is likewise an automorphic Schur ring associated to the subgroup of inner automorphisms $\text{Inn}(G)$.

Let $H\subseteq G$ and let $\S$ be a Schur ring over $G$. We say that 
$H$ is an \emph{$\S$-subset} 
if $H$ is a union of primitive sets of $\S$. 
We say 
$H$ is an \emph{$\S$-subgroup} 
if $H$ is both a subgroup of $G$ and an $\S$-subset. Note that the family of all $\S$-subsets forms a sublattice of the power set of $G$. Likewise, if $\L(G)$ denotes the lattice of all subgroups of $G$, then the family of all $\S$-subgroups form a sublattice of $\L(G)$. 

We say a Schur ring is \emph{primitive} if the only $\S$-subgroups are $1$ and $G$. The trivial Schur ring is primitive and all Schur rings over $\z_p$, where $p$ is a prime, are necessarily primitive. Wielandt \cite{Wielandt64} showed that these are the only primitive Schur rings over cyclic groups.

For a Schur ring $\S$ and an $\S$-subgroup $H$, let $\S_H:=\S\cap H$ be the Schur ring over $H$ associated to the partition of $H$ given by 
\[\D(\S_H) = \{ C\in \D(\S) \mid C\subseteq H\}.\] We say that a Schur ring $\T$ is a \emph{Schur subring} of $\S$ if $\T=\S_H$ for some $\S$-subgroup $H$. As the subgroups of $\z_n$ are uniquely determined by its order, if $\S$ is a Schur ring over $\z_n$ and $H$ is an $\S$-subgroup, then $\S_{|H|}:=\S_H$.

For automorphic Schur rings $\S=G^\H$, the $\S$-subgroups are exactly the $\H$-invariant subgroups of $G$, which includes all characteristic subgroups. Hence, if $\S$ is an automorphic Schur ring over $G$ with $\S$-subgroup $K$, then $S_K = K^\H$ which is necessarily automorphic itself.


For any group homomorphism $\varphi: G \to H$, this map lifts to a map on group algebras $\varphi : \QG \to \QH$ by the rule $\varphi\left(\sum_{g} \alpha_gg\right) = \sum_{g} \alpha_g\varphi(g)\in \QH$. Let $\S$ be a Schur ring over $G$. If $\ker\varphi$ is an $\S$-subgroup, then $\varphi(\S)$ is a Schur ring over $\varphi(G)$. In particular, the associated partition of $\varphi(\S)$ is \[\D(\varphi(\S)) = \{\varphi(C) \mid C\in \D(\S)\}.\]  If $\T$ is a Schur ring over $H$ such that $\varphi(G)$is a $\T$-subgroup and $\T_{\varphi(G)}=\varphi(\S)$, then $\varphi : \S\to \T$ induces a homomorphism between Schur rings. In the case that $\varphi : K\hookrightarrow
 G$ is the natural inclusion map, then $\varphi : \S_K \hookrightarrow
 \S$ becomes the inclusion map on Schur rings.

Suppose $G$ is a direct product of two groups, say $G=H\times K$. Let $\S$ and $\T$ be Schur rings over $H$ and $K$, respectively. Then the \emph{direct product}\footnote{These Schur rings are also known as tensor products, as in \cite{Tamaschke70}.} of Schur rings, denoted by $\S\times\T$, is given by the partition 
\[\D(\S\times\T) = \{CD\mid C\in \D(\S), D\in \D(\T)\},\] where $CD$ is viewed as a subset of $G$. By construction, $H$ and $K$ are both $\S\times \T$-subgroups. In fact, $(\S\times\T)_H = \S$ and $(\S\times \T)_K=\T$.  The direct product is the smallest Schur ring over $G$ with this property. The usual group projections $\pi_1 : H\times K\to H$ and $\pi_2: H\times K\to K$ induce Schur ring projections $\pi_1 : \S\times \T\to \S$ and $\pi_2:\S\times \T\to \T$.


If ${h}\in \D(\S)$ for some Schur ring $\S$ over $G$, then $hC\in \D(\S)$ for all $C\in \D(\S)$. Likewise, the collection of singletons in $\D(\S)$ forms an $\S$-subgroup $H$ of $G$. Of course, $\S_H=H$, that is, $\S_H$ is discrete. Furthermore, if $G=H\times K$ and $\S$ is a Schur ring over $G$ such that $H,K$ are $\S$-subgroups and $\S_H$ is discrete, then $\S=\S_H\times \S_K=H\times \S_K$.\footnote{This fact was proven in \cite{Ponomarenko16} in the case of abelian groups, but the commutativity assumption can be dropped with no loss.}

For a Schur ring $\S$, we say that an $\S$-subgroup $K$ is \emph{normal} if $K\trianglelefteq G$. Let $\varphi : G\to G/K$ be the natural map. Then we call the Schur ring $\varphi(\S)$ the \emph{quotient Schur ring} of $\S$ over $G/K$ and denote it as $\S/K$. 

Conversely, if $\varphi:G\to H$ is a group homomorphism and $\T$ is a Schur ring over $H$, then define \[\varphi^{-1}(\T) := \Span\{\varphi^{-1}(C) \mid C\in \D(\T)\}.\] It is elementary to show that $\varphi^{-1}(\T)$ is closed under $*$ and multiplication, but, the class containing the identity is $\ker\varphi$, which, in general, is not $1$. As such, we call $\varphi^{-1}(\T)$ a \emph{pre-Schur ring}. Note that all primitive sets in $\varphi^{-1}(\T)$ are unions of cosets of $\ker\varphi$. 

We say that $[K,H]$ is an \emph{$\S$-section} if $1\le K\le H\le G$, $K\trianglelefteq G$, and $H, K$ are $\S$-subgroups. We say that a section $[K,H]$ is \emph{proper} if $1<K\le H< G$, and we say that $[K,H]$ is \emph{trivial} if $K=H$.  As all subgroups of $\z_n$ are normal and uniquely determined by their orders, we shall denote the section $[\z_d, \z_e]$ simply as $[d,e]$. 

We say that $\S$ is \emph{wedge-decomposable} if there exists a proper $\S$-section $[K,H]$ such that for every $\S$-class $C$ either $C\subseteq H$ or $C$ is a union of cosets of $K$. Otherwise, $\S$ is \emph{wedge-indecomposable}. 

Let $U=[K,H]$ be a proper section of $G$. We say a pair of Schur rings $\S$ and $\T$ over $H$ and $G/K$, respectively, are \emph{wedge-compatible} if $K$ is an  $\S$-subgroup, $H/K$ is a $\T$-subgroup, and $\S/K = \T_{H/K}$. Then the \emph{wedge product}\footnote{The wedge product was first introduced by Leung and Man \cite{LeungI}. In \cite{Ponomarenko02}, Evdokimov and Ponomarenko independently introduce the similar notion of a \emph{(generalized) wreath product}.} of two wedge-compatible Schur rings $\S$ and $\T$, denoted $\S\wedge_U \T$, is given by the rule $\S\wedge_U\T := \S + \varphi^{-1}(\T)$, where $\varphi : G \to G/K$ is the natural map. Note that 
\[\D(\S\wedge_U\T) = \D(\S) \cup \{\varphi^{-1}(C) \subseteq G\smallsetminus H \mid C\in \D(\T)\}.\] Note that by construction $U$ is an  $\S\wedge_U \T$-section, $(\S\wedge_U \T)_H = \S$, and $(\S\wedge_U\T)/K = \T$. The wedge product is, in fact, the smallest Schur ring over $G$ with this property. Furthermore, Schur rings over $G$ can be factored as a wedge product if and only if the Schur ring is wedge-decomposable.

With a trivial section wedge-compatibility is automatic. Thus, if $G$ is any group extension of $Q$ by $N$, then we may form the wedge product $S\wedge_{[N,N]} \T$ over $G$ where $\S$ and $\T$ are any Schur rings over $N$ and $Q$, respectively. In this case, the subscript is omitted and the product is denoted simply as $\S\wedge \T$. 

Finally, we say that a Schur ring over $G$ is \emph{traditional} if the Schur ring can be recursively built using automorphic and trivial Schur rings over sections of $G$ via the operations of direct products and wedge products. In other words, a Schur ring is traditional if it belongs to one of four Schur ring families: trivial Schur rings, automorphic Schur rings, direct products, or wedge products. Leung and Man show in \cite{LeungII} the following important classification of Schur rings over cyclic groups.

\begin{Thm}[The Fundamental Theorem of Schur Rings over Cyclic Groups]\label{thm:fundamental} All Schur rings over the finite\footnote{In \cite{InfiniteI}, this result is extended to the infinite cyclic group, among others.} cyclic group $\z_n$ are traditional.
\end{Thm}

\section{Counting Schur Rings}\label{sec:counting}
In order to compute $\Omega(n)$, we use the above Fundamental Theorem and count the number of Schur rings of each of the four families of traditional rings, with particular attention on what conditions cause them to overlap. 

For any $n$, there is exactly one trivial Schur ring. As it is primitive, the trivial Schur rings $\z_n^0$ will never be wedge-decomposable nor factorable as a direct product, since both products require a proper subgroup which is absent for primitive Schur rings. In the case that $n=p$, $\z_p^0 = \z_p^{\Aut(\z_p)}$. If $n$ is not prime, then elements of different orders in $\z_n^0$ are fused together, a feature impossible in any automorphic Schur ring. Hence, the trivial Schur ring is automorphic if and only if the group has prime order.  Thus, the trivial Schur ring will contribute a single count to $\Omega(n)$ and is distinct from the other three traditional families for composite orders. 

Let $n=ab$ be a unitary factorization, that is, $\gcd(a,b)=1$. The number of direct product Schur rings over $\z_{n}$ with respect to this factorization will be $\Omega(a)\Omega(b)$, since each such Schur ring has the form $\S\times\T$, where $\S$ and $\T$ are Schur rings over $\z_a$ and $\z_b$, respectively. One can enumerate all direct product Schur rings over $\z_n$ by enumerating all unitary factorizations of $n$. As the direct product is an associative, commutative operator on groups, a typical inclusion-exclusion argument is necessary when $n$ has at least three prime divisors. For example, the number of direct product Schur rings over $\z_{pqr}$ would be $\Omega(p)\Omega(qr) + \Omega(q)\Omega(pr) + \Omega(r)\Omega(pq) - 2
\Omega(p)\Omega(q)\Omega(r).$

 Let $H$ and $K$ be groups and let $\H\le \Aut(H)$ and $\K\le \Aut(K)$. If $G=H\times K$, then we may naturally view $\H\times \K$ as a subgroup of $\Aut(G)$ by the following rule. If $\sigma\in \H$ and $\tau\in \K$, then define the map $\sigma\times \tau : H\times K \to H\times K$ as $(h,k)^{\sigma\times \tau} = (h^\sigma, k^\tau)$. If follows that $\S\times \T$ is automorphic if and only if $\S$ and $\T$ are automorphic.


The automorphic Schur rings over $\z_n$ are in Galois correspondence with the subgroups of $\Aut(\z_n)$.\footnote{See Equation (3.1) in \cite{Counting}.}  Let $U(n)$ denote the set of units of the finite ring $\Z_n$, that is, integers modulo $n$. Thus, $U(n)$ is the set of integers coprime to $n$. It is clear that $\Aut(\z_n)\cong U(n)$ and 
\[U(n) = \prod_{i=1}^r U(p_i^{e_i}),\] where $n=\prod_{i=1}^r p_i^{e_i}$ is the prime factorization of $n$. Hence, counting automorphic Schur rings over $\z_n$ is equivalent to computing $|\L(U(n))|$.

Note, $U(n)$ is cyclic if and only if $n=p^k$ or $n=2p^k$. In this case $\L(U(n))$ decomposes into a tower of sublattices, which we call \emph{layers}, where each layer is lattice-isomorphic to $U(p)$ and sits on top of the previous layer by degree $p$ extensions. Additionally, $U(p)$ is isomorphic to the divisor lattice associated to $p-1$. If $x$ is the number of divisors of $p-1$, then the number of automorphic Schur rings over $\z_{p^k}$ or $\z_{2p^k}$ is  $(k-1)x$. 


The general problem of enumerating automorphic Schur rings over $\z_n$ is much more difficult, as every abelian group is a subgroup of $U(n)$ for some sufficiently large $n$. The problem of counting the number of subgroups of an arbitrary abelian group is a well studied problem in the literature, for example \cite{OhJuMok}, \cite{Petrillo}, and \cite{Toth14}, which essentially all derive from a theorem of Goursat \cite{Goursat}. Our consideration of this problem will follow the method of C\u{a}lug\u{a}reanu \cite{Calug}. We say that two sections $U=[K,H]$  and $U'=[K',H']$ over $G$ and $G'$, respectively, are \emph{isomorphic}, if $H/K\cong H'/K'$. If $G=H\times K$, we say a subgroup $D$ is \emph{diagonal} if $D\in \L(G)\smallsetminus (\L(H)\times \L(K))$. As shown in \cite{Calug}, diagonal subgroups of $G=H\times K$ correspond exactly to automorphisms between isomorphic, non-trivial sections of the lattices $\L(H)$ and $\L(K)$. If $\gcd(|H|, |K|)=1$ then it follows that $\L(H\times K) \cong \L(H)\times \L(K)$. Thus, it suffices to consider the case that $H$ and $K$ are both cyclic $p$-groups.  We illustrate C\u{a}lug\u{a}reanu's technique below.

\begin{Exam}\label{exam:4x4} In $\z_4\times\z_4$ there are 6 diagonal subgroups. To see this, note that the sections in $\z_4$ are $[1,2]$, $[2,4]$, and $[1,4]$. The first two sections are isomorphic to $\z_2$. Since there is only one automorphism over $\z_2$, $1\cdot2\cdot2=4$ of the diagonal subgroups arise from the combinations of $[1,2]$ and $[2,4]$. The last $2\cdot1\cdot1=2$ come from there being two automorphisms over $\z_4$ and the single combination of $[1,4]$ and itself. Given that $|\L(\z_4)|=3$, this shows that $|\L(\z_4\times\z_4)| = 3^2+6=15$.
\end{Exam}

\begin{Lem}[\cite{Calug}]\label{lem:latticeabelian} The number of subgroups of $\z_{p^k}\times \z_{p^\ell}$ is given as
\[\left|\L\middle(\z_{p^k}\times \z_{p^\ell}\middle)\right| = \sum_{j=0}^{\min(k,\ell)} \phi(p^j)(k-j+1)(\ell-j+1),\] where $\phi$ denotes Euler's totient function.
\end{Lem}

Note that for $\z_4\times \z_4$, we see that 
\[|\L(\z_4\times \z_4)| = 1(3)(3) + 1(2)(2) + 2(1)(1) = 15,\] which agrees with \examref{exam:4x4}. 

If $\S$ is wedge-decomposable with section $[K,H]$, then $\S$ is likewise decomposable for the section $[K',H]$ where $K'\trianglelefteq G$ and $K'\le K$. Thus, it is advantageous to choose $K$ to be a minimal normal subgroup to avoid unnecessary over-counting of the same wedge product that can be formed from distinct sections. In \cite{Counting}, the section $[p, p^j]$ was used, where $j$ was allowed to vary, for this very reason. For general $n$, when the subgroup $H$ has multiple prime divisors, there are multiple minimal subgroup choices for $K$ and an  inclusion-exclusion argument is again necessary.


When enumerating direct product Schur rings, it is important to observe that if $G=H\times K$ is a group and $\S=\S_H\times \S_K$ is a direct product Schur ring over $G$ such that $\S_H$ is wedge-decomposable, then $\S$ is itself wedge-decomposable.\footnote{This is immediate consequence of \cite[Theorem 4.1]{Ponomarenko13}, but an elementary argument using sections is also possible allowing for the assumption that $G$ is cyclic to be removed.} Hence, it suffices to count only those with wedge-indecomposable direct factors, as any complete enumeration of wedge products will contain the rest. The indecomposable Schur rings are necessarily then trivial, automorphic, or direct products of indecomposable factors. Hence, the trivial and indecomposable automorphic Schur rings will be the atomic building blocks for all higher Schur rings over cyclic groups. 

For $\z_{p^k}$, the indecomposable Schur rings include the trivial ring $\z_{p^k}^0$ and the top-layer automorphic rings,\footnote{Such Schur rings were considered in \cite{Ponomarenko02} using a different approach. In particular, they showed that a Schur ring over a cyclic group is indecomposable if and only if its radical is trivial.  In the case of orbit Schur rings, it is shown that this occurs exactly when the associated automorphic group's order is coprime to to $p$. Most arguments involving the representation $\omega$ in this manuscript could be alternatively argued using this radical.} that is, those automorphic Schur rings that correspond to the subgroups of $U(p)$ viewed as a subgroup of $U(p^k)$. To see this, consider the representation $\omega : \Q[\z_n] \to \Q(\zeta_n)$ afforded by the rule $z\mapsto \zeta_n$, where $\zeta_n$ is a complex, primitive $n$th root of unity and $\Q(\zeta_n)$ is the associated cyclotomic field.\footnote{Because of this representation by cyclotomic fields, automorphic Schur rings over a cyclic group are often called \emph{cyclotomic} in the literature. See \cite{Ponomarenko02, Muzychuk09, Ponomarenko16}.}  For a Schur ring $\S$, $\omega(\S)$ is necessarily a subfield of $\Q(\zeta_n)$ and, as $U(n)$ is the Galois group of $\Q(\zeta_n)$ over $\Q$, the correspondence between automorphic Schur rings and the subgroups of $U(n)$ is identical to the Galois correspondence of subfields of $\Q(\zeta_n)$ with subgroups of $U(n)$. Note that\footnote{See \cite[Equations (3.1)--(3.3)]{Counting}. Direct products were not considered here because $\z_{p^k}$ has not such factorization.} $\omega(\z_n^0) = \Q$, $\omega(\z_n^\H) = \Q(\zeta_n)^\H$, and 
$\omega(\S_H\wedge \S/K) = \omega(\S_H)$. Likewise, $\omega(\S_m) \subseteq \Q(\zeta_m)$. In particular, a wedge-decomposable ring cannot map into the \emph{top-layer} of $\Q(\zeta_n)$, those subfields of $\Q(\zeta_n)$ which are not subfields of $\Q(\zeta_d)$ for any proper divisor $d$ of $n$. This implies that those rings which do map to the top-layer are indecomposable, as claimed.

We can also consider the $\omega$-image of a direct product. Recall that $\S\times \T$ is the compositum ring of $\S$ and $\T$. As such, $\omega$ will map $\S\times\T$ onto the compositum field of $\omega(\S)$ and $\omega(\T)$ in $\Q(\zeta_n)$, that is, 
\begin{equation}\label{eq:omegaDirect} \omega(\S\times\T) = \omega(\S)\vee \omega(\T) \cong \omega(\S)\otimes_\Q\omega(\T).\end{equation}



As mentioned above, if $\S = \S_{H'}\wedge_{[K',H']} \S/K'$ is wedge-decomposable, then $\omega(\S) = \omega(\S_{H'})$. If $\S_{H'}$ is itself wedge-decomposable over the section $[K'',H'']$, then $\omega(\S) = \omega(\S_{H''})$. As these groups are finite, this process will eventually terminate, and there exists an $\S$-subgroup $H$ such that $\omega(\S)=\omega(\S_H)$ and $\S_H$ is indecomposable. Let $H$ be the maximal such subgroup. Then $\S_H = \z_{d}^\H\times \prod_{i=1}^k \z_{d_i}^0$, where any of the divisors $d_0, d_1, \ldots$ could be 1. This will be the unique, maximal wedge-indecomposable Schur subring of $\S$, which we call the \emph{wedge-core} of $\S$. Two Schur rings with distinct cores are necessarily distinct. If $\S$ is itself indecomposable, then it is equal to its own core. Every Schur ring $\S$ over a non-trivial cyclic group will have a non-trivial core as the core will contain all the Schur subrings over the minimal $\S$-subgroups. We will enumerate the possible cores of Schur rings to avoid situations when two different wedge decompositions give rise to the same partition of $\z_n$.

Let $\Omega(n, \S)$ denote the number of Schur rings over $\z_{n}$ that contain the Schur ring $\S$ as its wedge-core. We let $\Omega(n,d)$\footnote{Note that this notation has a slightly different meaning in \cite{Counting}, for which the present paper generalizes the situation considered in the former.} abbreviate $\Omega(n, \z_{d})$. If $\S\wedge_{[K,H]} \T$ is a wedge product, then it must be that $\S/K = \T_{H/K}$. If we keep $\S$ fixed, then $\T$ may be any Schur ring over $G/K$ so long as the Schur subring $\T_{H/K}$ is exactly $\S/K$. Of course, if $\S$ is an indecomposable Schur ring over $\z_{n}$, then $\Omega(n,\S)=1$.

\begin{Lem}\label{lem:primcount} If $\S$ is a primitive Schur ring over $\z_d$ and $d\mid n$, then $\Omega(n,\S)=\Omega(n/d)$.
\end{Lem}
\begin{proof}
If $\S\wedge_{U} \T$ is a Schur ring over $\z_n$ with section $U=[K,H]$, then $K=H$ by the primitivity of $\S$. As this is a trivial section, wedge-compatibility is automatic. Thus, $\T$ could be any Schur ring over $\z_n/\z_d \cong \z_{n/d}$.
\end{proof}

\section{Examples of Enumerating Schur Rings over Cyclic Groups}\label{sec:proof}
Using the strategies from the previous section, we consider some examples of enumerating Schur rings over cyclic groups.  The first two can be found in \cite{Counting}. 

\begin{Exam} \label{exam:oddprime}
Let $p$ be an odd prime and let $x$ be the number of divisors of $p-1$. Then $\Omega(p)=x$, as all Schur rings over $\z_p$ are automorphic, including the trivial one, and these Schur rings are in one-to-one correspondence with the divisors of $p-1$.

Let $n=p^k$ for $k>1$. Then the indecomposable Schur rings over $\z_{p^k}$ are $\z_{p^k}^0$ and those $x$-many top-layer automorphic Schur rings. In \cite[Lemma 4.4]{Counting}, we see that for Schur rings $\S$ and $\T$ over $\z_{p^k}$ whose cores are indecomposable automorphic Schur subrings of the same order $p^j$, $\Omega(p^k, \S) = \Omega(p^k, \T) = \Omega(p^k,p^j)$. This equality is based upon identical recursive relations on these different functions. Note that for every order $p^j$, if $\Omega(p^k,\S)=\Omega(p^k,p^j)$ then there are exactly $x=\Omega(p)$ many options for the core of $\S$, namely the $x$ indecomposable automorphic Schur rings over $\z_{p^j}$.  Hence, 
\[\Omega(p^k) = \sum_{j=1}^k (\Omega(p^k, \z_{p^j}^0) + x\Omega(p^k, p^j)) - \Omega(p^k,p).\] Of course, $\Omega(p^k, \z_{p^j}^0) = \Omega(p^{k-j})$ and $\Omega(p^k,p)=\Omega(p^{k-1})$ by \lemref{lem:primcount}. Likewise, $\Omega(p^k,p^k)=1$, as observed in the paragraph before \lemref{lem:primcount}. Finally, by observing the recursive relation $\Omega(p^k, p^j) = \sum_{i=j}^k \Omega(p^{k-1}, p^{i-1})$ for $1<j<k$, we may unravel the unravel the equation above into
\[\Omega(p^k) = x\Omega(p^{k-1}) + \sum_{j=2}^k(c_{j-1}x + 1)\Omega(p^{k-j}),\] where $c_k = \dfrac{1}{k+1}\dbinom{2k}{k}$ is the $k$th Catalan number. We illustrate this formula for small powers of $p$. Of course, $\Omega(p)=x$. For $n= p^2$, 
\begin{multline*}
\Omega(p^2) = x\Omega(p^2,p) + \Omega(p^2,\z_{p^2}^0)+x\Omega(p^2,p^2) = x\Omega(p) + \Omega(1) + x\cdot 1 = x^2+x+1.
\end{multline*}
For $n= p^3$, 
\begin{multline*}
\Omega(p^3) = x\Omega(p^3,p) + \Omega(p^3,\z_{p^2}^0)+x\Omega(p^3,p^2) + \Omega(p^3,\z_{p^3}^0)+x\Omega(p^3,p^3)\\ = x\Omega(p^2) + \Omega(p) + x(\Omega(p^2,p) + \Omega(p^3,p^3)) + \Omega(1) + x\cdot 1\\ = x(x^2+x+1) + x + x(\Omega(p) + \Omega(1)) + 1+x\\ = x^3+2x^2+4x+1. \qedhere
\end{multline*}
\end{Exam}

\begin{Exam} Let $p=2$. Of course, $\Omega(2)=1$, namely $\z_2$ itself. Similarly, $\Omega(4)=3$, namely $\z_4^0$, $\z_2\wedge \z_2$, and $\z_4$. For $k>2$, we note that $U(2^k) \cong \z_2\times \z_{2^{k-1}}$. Unlike the odd prime case, this automorphism group is non-cyclic and contains diagonal subgroups. As such, the number of Schur rings which have $\z_{2^j}$ as its core will be distinct from the number of Schur rings which have $\z_{2^j}^\pm$, the automorphic Schur ring corresponding to the inversion map, as its core. Hence,
\[\Omega(2^k) = \sum_{j=2}^k \Omega(2^k, \z_{2^j}^0) + \Omega(2^k, 4) + \sum_{j=3}^k (\Omega(2^k, 2^j) + 2\Omega(2^k, \z_{2^j}^\pm).\] Some of the calculations are similar to the previous case, by virtue of \lemref{lem:primcount}, namely:  $\Omega(2^k, 2^k) = \Omega(2^k, \z_{2^k}^\pm) =1$, $\Omega(2^k, \z_{2^j}^0) = \Omega(2^{k-j})$, and $\Omega(2^k, 2) = \Omega(2^{k-1})$, but the remaining recursive relations are much more complicated, namely: $\Omega(2^k, 4)=\Omega(2^{k-1}) - \sum_{j=1}^{k-2} \Omega(2^{k-2}, \z_{2^j}^0)$, $\Omega(2^k, 2^j)=\sum_{i=j-1}^{k-1} \Omega(2^{k-1}, 2^i)$, $\Omega(2^k, \z_{2^j}^\pm)=\Omega(2^{k-1},\z_{2^{j-1}}^\pm)+2\sum_{i=j}^{k-1}\Omega(2^{k-1}, \z_{2^i}^\pm)$, and $\Omega(2^k, \z_{8}^\pm)=\Omega(2^{k-1}, 4)+2\sum_{i=3}^{k-1}\Omega(2^{k-1}, \z_{2^i}^\pm).$ Unraveling the recursive relations, we have
\[
\Omega(2^k) = \sum_{j=1}^3 2^j\Omega(2^{k-j}) - (c_{k-1} + s_{k-1}) + \sum_{j=4}^k\left(c_{j-1} + s_{j-1}-\sum_{i=1}^{j-3}(c_i+s_i)\right)\Omega(2^{k-j})\] where $c_k = \dfrac{1}{k+1}\dbinom{2k}{k}$ is the $k$th Catalan number and $s_k = \dsum_{j=0}^k \frac{1}{j+1}\binom{2j}{j}\binom{k+j}{2j}$  is the $k$th Schr\"{o}der number.
\end{Exam}

The following two examples can be found in \cite{CountingIII}.  

\begin{Exam} Let $p$ and $q$ be distinct primes. For the case $n=pq$
we do not need to consider any direct products, as they are all automorphic (note all Schur rings over $\z_p$ and $\z_q$ are automorphic). As the only proper section over $\z_{pq}$ are trivial, the only wedge products are over the sections $[p,p]$ and $[q,q]$, which give, by \lemref{lem:primcount}, $\Omega(p)\Omega(q)$ and $\Omega(q)\Omega(p)$ many Schur rings, respectively. Thus, 
\[\Omega(pq) = 2\Omega(p)\Omega(q) + |\L(U(pq))|+1.\]
Let $p=\prod_{i=1}^n r_i^{k_i} +1$ and $q=\prod_{i=1}^n r_i^{\ell_i} +1$, where $\{r_1, r_2,\ldots, r_n\}$ is a list of distinct primes. Combining the above formula with \lemref{lem:latticeabelian}, we have 
\[\Omega(pq) = \prod_{i=1}^n \sum_{j=0}^{\min(k_i,\ell_i)} \phi(r_i^{j})(k_i-j+1)(\ell_i-j+1) + 2\prod_{i=1}^n (k_i+1)(\ell_i+1) + 1,\] where $\phi$ denotes Euler's totient function.  We illustrate this formula for the cases $n= 2p, 3p, 5p$. If $p\neq 2$ is a prime and $x$ is the number of divisors of $p-1$, then 
\[\Omega(2p)=3x+1.\] If $p\neq 3$ is a prime such that $p=2^ka+1$ where $a$ is odd, then 
\[\Omega(3p) = \left(\dfrac{7k+6}{k+1}\right)x + 1.\] If $p\neq 5$ is a prime such that $p=2^ka+1$ where $a$ is odd, then 
\[\Omega(5p) = \left(\dfrac{13k+7}{k+1}\right)x + 1. \qedhere\] 
\end{Exam}

\begin{Exam} Let $p$ be an odd prime such that $p = 2^ka+1$, where a is an odd integer and $x$ the number of divisors of $p-1$. For the case $n=4p$, the indecomposable automorphic Schur rings will be in one-to-one correspondence with the subgroups of $U(4p) \cong \z_2\times \z_{p-1}$. To see this, we note that $U(4p)$ contains exactly two layers lattice-isomorphic to $U(p)$ with some diagonal subgroups sitting in between  the layers. The top-layer subgroups of $U(4p)$ correspond to direct products $\z_4\times \S$, where $\S$ is a Schur ring over $\z_p$, with respect to the usual correspondence of automorphic Schur rings and automorphic subgroups. On the other hand, the direct products of the form $\z_4^0\times \S$ are indecomposable and can be made to correspond the the subgroups of the bottom-layer of $U(4p)$. Those automorphic Schur rings which correspond to diagonal subgroups of $U(4p)$ necessarily must be wedge-indecomposable. Hence, if we combine together these three sets of Schur rings, we establish a one-to-one correspondence between them and the subgroups of $U(4p)$. Therefore,
\begin{multline*} \Omega(4p) = \Omega(4p,2) + x\Omega(4p,p) + \Omega(4p, \z_{2p}^0) + x\Omega(4p, 2p) + \Omega(4p,\z_4^0)+\Omega(4p,4)\\ + \Omega(4p, \z_{4p}^0) + |\L(U(4p))| = \Omega(2p) + x\Omega(4) + \Omega(2) + x(\Omega(2p,p)+\Omega(2p,2p)+\Omega(4,2)\\
+\Omega(4,4)-\Omega(2)) +\Omega(p)+ (\Omega(2p,2) + x\Omega(2p,2p)) + \Omega(1) + |\L(U(4p))|\\
=(3x+1)+3x+1+x(\Omega(2)+\Omega(1)+\Omega(2) + 1 - 1) + x + (\Omega(p)+x\cdot1) + 1 + |\L(U(4p))|\\
=12x+3+ |\L(U(4p))|.
\end{multline*}
To count the automorphic Schur rings, we compute 
\[|\L(U(4p))|  = |\L(\z_2\times \z_{2^k})||\L(\z_a)| = (2(k+1)+k)\left(\dfrac{x}{k+1}\right)=\dfrac{3k+2}{k+1}x.\]
Therefore, \[\Omega(4p)=\frac{15k+14}{k+1}x+3. \qedhere\]  
\end{Exam}



Finally, we present a new example to illustrate these enumeration techniques. 

\begin{Exam}
For $n=2p^2$,
\begin{multline*} \Omega(2p^2) = \Omega(2p^2, 2) + x\Omega(2p^2,p) + \Omega(2p^2, \z_{2p}^0) + x\Omega(2p^2, 2p)+\Omega(2p^2, \z_{p^2}^0) +x\Omega(2p^2, p^2) + \Omega(2p^2, \z_{2p^2}^0)\\ + \Omega(2p^2, \z_2\times \z_{p^2}^0) + x\Omega(2p^2, 2p^2) = \Omega(p^2) + x\Omega(2p) + \Omega(p) + x(\Omega(p^2,p) + \Omega(p^2,p^2) + \Omega(2p,2) + x\Omega(2p,2p)\\ - \Omega(2p,2p)) + \Omega(2) + x(\Omega(2p,2)+\Omega(2p, 2p)) + \Omega(1) + x\cdot 1 = (x^2+x+1)+x(3x+1) + x + x(\Omega(p)+\Omega(1)\\+\Omega(p)+ x\Omega(1)-\Omega(p)) + 1 + x(\Omega(2)+1) + 1 + x = 6x^2+7x+4. 
\end{multline*} By a similar computation that is omitted here, we have 
\[\Omega(2p^3) = 10x^3+21x^2+31x+6. \qedhere\]
\end{Exam}

\section{Conclusion}
We can extrapolate from the above examples a general strategy for counting Schur rings over cyclic groups. We begin by identifying the indecomposable Schur rings for all divisors of the order $n$. Once this is complete, we proceed to enumerate all wedge products choosing as the left factor only indecomposable Schur rings and choosing as sections only those of the form $[K,H]$, where $K$ is the minimal subgroup of the left wedge-factor. This is trivial for primitive Schur rings and the Principle of Inclusion-Exclusion is necessary when distinct minimal subgroups are present. For every left factor in $\S\wedge_{[d,e]} \T$, there are $\Omega(n/d, \pi(\S))$ options for $T$. The sum of these mutually exclusive cases gives $\Omega(n)$. 

This general strategy comes with two major obstacles. First, it requires a strong understanding of the recursive nature of the function $\Omega(n,\S)$. While ad hoc arguments are used here to handle the examples considered herein, the potential complexity of $\Omega(n,\S)$ is seen clearly in \cite{Counting}. Second, it requires we be able to effectively enumerate the indecomposable Schur rings. While primitive Schur rings are easy to identify for cyclic groups and direct products are indecomposable only if their direct factors are, the indecomposable automorphic Schur rings are a greater challenge. As we saw throughout, we can identify the indecomposable automorphic rings using the lattice $\L(U(n))$, but this lattice becomes increasingly more difficult as the rank of $U(n)$ increases. C\u{a}lug\u{a}reanu's technique provides an effective method of counting subgroups of an abelian group of rank 2, but it becomes far less effective for rank 3 and beyond. An explicit formula for enumerating subgroups of an abelian groups of rank 3 is fairly recent (see the aforementioned references for details), and, at the time of writing, any explicit formula beyond rank 3 has yet to be discovered. As such, any explicit formula for enumerating Schur rings over cyclic groups is unlikely without an explicit formula for counting subgroups of abelian groups.

\bibliographystyle{plain}
\bibliography{Srings}
\end{document}